\documentclass{amsart}
\usepackage[utf8]{inputenc}

\usepackage{amssymb,amsmath, amsthm, amsfonts}
\usepackage{mathrsfs}
\usepackage[ruled,linesnumbered]{algorithm2e}

\usepackage[margin=1in]{geometry}
\usepackage{graphicx, subcaption}
\usepackage[dvipsnames]{xcolor}
\usepackage{tikz-cd}
\usepackage[doi=true,isbn=false,url=false,eprint=false,style = alphabetic]{biblatex}
\usepackage{multirow}
\usepackage{nicematrix}

\def\spann{\mathop{\rm span}\nolimits}

\DeclareMathOperator*{\supp}{supp}

\def\trace{\mathop{\rm trace}}

\def\cone{\mathop{\rm cone}}

\def\diag{{\rm diag}}
\def\idx{{\rm idx}}

\newcommand{\bbR}{\mathbb{R}}

\newcommand{\cA}{\mathcal{A}}
\newcommand{\cB}{\mathcal{B}}

\newcommand{\cR}{\mathcal{R}}

\newcommand{\cM}{\mathcal{M}}

\newcommand{\cS}{\mathcal{S}}

\newcommand{\cG}{\mathcal{G}}
\newcommand{\cQ}{\mathcal{Q}}
\newcommand{\cK}{\mathcal{K}}

\newcommand{\Cop}{\mathcal{C}op}
\newcommand{\SPN}{\mathcal{SPN}}
\newcommand{\PSD}{\mathcal{PSD}}
\newcommand{\DNN}{\mathcal{DNN}}
\newcommand{\CP}{\mathcal{CP}}

\newtheorem{theorem}{Theorem}[section]
\newtheorem{corollary}[theorem]{Corollary}
\newtheorem{lemma}[theorem]{Lemma}
\newtheorem{proposition}[theorem]{Proposition}
\newtheorem{definition}[theorem]{Definition}
\theoremstyle{definition}
\newtheorem{example}[theorem]{Example}

\newtheorem{remark}[theorem]{Remark}

\newtheorem{observation}[theorem]{Observation}

\usepackage{hyperref}
\hypersetup{
    colorlinks=true,
    linkcolor = blue,
    citecolor = blue,
    urlcolor = blue
}

\usepackage{spectralsequences} 
\usepackage{tikz}
\usetikzlibrary{decorations.pathreplacing,matrix}
\tikzset{ 
table/.style={
  matrix of math nodes,
  row sep=-\pgflinewidth,
  column sep= 100\pgflinewidth,
  nodes={rectangle,text width=3em,align=center},
  text depth=1.25ex,
  text height=2.5ex,
  nodes in empty cells,
  ampersand replacement=\&
}
}

\usepackage{pdfpages}
\usepackage{enumitem}

\title{Copositive Matrices with Ordered Off-Diagonal Entries}

\author{Grigoriy Blekherman}\address{Grigoriy Blekherman, School of Mathematics, Georgia Institute of Technology}\email{greg@math.gatech.edu}

\author{Santanu S. Dey}\address{Santanu S. Dey, H. Milton Stewart School of Industrial and Systems Engineering, Georgia Institute of Technology}\email{santanu.dey@isye.gatech.edu}

\author{Alex Dunbar}\address{Alex Dunbar, School of Mathematics, Georgia Institute of Technology}\email{adunbar30@gatech.edu}

\author{Burak Kocuk}\address{Burak Kocuk, Industrial Engineering Program, Sabancı University}\email{burakkocuk@sabanciuniv.edu}

\subjclass[2020]{90C20, 90C22, 15B48}
\keywords{Copositive Matrices, SPN Matrices, Standard Quadratic Programs, Threshold Graphs}

\begin{document}

\begin{abstract}
    We study copositive matrices which admit a decomposition into a sum of a positive semidefinite matrix and a matrix with nonnegative entries. Our main result shows that if the off-diagonal entries of a copositive matrix are nondecreasing in rows and in columns, then it admits such a decomposition. We apply this result to study optimization of quadratic forms over the standard simplex. As a corollary, we obtain that a natural relaxation of this problem is tight when the objective function is separable, resolving an open question from~\cite{dey_convexification_2025-1}.
\end{abstract}

\maketitle

\section{Introduction}

A symmetric $n\times n$ matrix $A$ is \emph{copositive} if the quadratic form $x^\top A x$ takes only nonnegative values for $x$ in the nonnegative orthant. The collection of copositive $n\times n$ matrices forms a closed convex cone $\Cop_n$ inside the space of all $n\times n$ symmetric matrices. There has been recent interest in the cone of copositive matrices, as many optimization problems with quadratic objectives can be reformulated as conic optimization problems over the completely positive cone of matries, which are dual to conic optimization problems over the copositive cone \cite{bomze_copositive_2000,burer2009copositive,cifuentes2025sensitivity}. A classical example is the problem of minimizing a quadratic form over the unit simplex, also called a \emph{standard quadratic program}~\cite{bomze1998standard,bomze_copositive_2000,bomze2002solving,bomze2008new,chen2013sparse,saugol2015analysis,bomze2025tighter}, which  has an equivalent copositive conic reformulation. 

Checking copositivity of a matrix is co-NP-complete \cite{murty_np-complete_1987}. A more tractable inner approximation of the cone of copositive matrices is the cone of symmetric matrices which are the sum of a positive semidefinite matrix and a matrix with only nonnegative entries. Denoting this cone by $\SPN_n$, it is a foundational result that $\SPN_n = \Cop_n$ if and only if $n\leq 4$ \cite{diananda1962non}. For $n \geq 5$, there has been interest in determining conditions on copositive matrices which guarantees an SPN decomposition \cite{shaked-monderer_spn_2016,gokmen2022standard}.

We have two main goals in this paper. First, we develop new conditions, depending only on the off-diagonal entries of a copositive matrix, which guarantee the existence of an SPN decomposition. Secondly, we study the implications of such results on standard quadratic programming problems, since a natural approximation for a copositive reformulation is to replace the cone $\Cop_n$ with the cone $\SPN_n$.

Towards our first goal, we present a simple criterion which ensures that a copositive matrix admits an SPN decomposition as follows:
\begin{theorem}[Theorem \ref{thm:Mn_SPN} below]\label{thm:OrderedEntries_Intro}
If $A = (a_{i,j})\in \Cop_n$ and the off-diagonal entries of $A$ are nondecreasing in rows and columns, i.e. $a_{i,j} \leq a_{\ell,k}$ when $i\neq j, \ell\neq k,i \leq \ell,$ and $j \leq k$, then $A \in \SPN_n$.
\end{theorem}
Our proof relies on an inductive argument which applies to a broader class of matrices, described in Theorem~\ref{thm:MainTheoremGeneral}. Building on this inductive framework, we obtain further generalizations, including Theorem~\ref{thm:RelaxedOrdering}, in which the ordering assumptions on the off-diagonal entries are weakened: positive entries are required to be nondecreasing in rows and columns, while nonpositive entries need only satisfy a row-ordering condition determined by the location of the first positive entry in the matrix. As a consequence, we obtain exactness of the SPN relaxation for block-structured matrices in which one block has nonpositive off-diagonal entries, the off-diagonal block is nonpositive with nondecreasing rows, and the remaining block has ordered off-diagonal entries, as well as for matrices whose blocks overlap in a single row and column (Corollary~\ref{ex:BlockSign_SPN} and Proposition~\ref{prop:almost_block_Mn}). The matrices considered in Theorem~\ref{thm:OrderedEntries_Intro} also generate a larger collection of matrices for which copositivity implies the existence of an SPN decomposition, via orbits of a group action. Since copositivity and the existence of an SPN decomposition are both invariant under simultaneous permutation of rows and columns and under rescaling by diagonal matrices with positive diagonal, one can ask when a given symmetric matrix can be brought into the ordered setting by such operations. We show in Propositions~\ref{prop:Rescale_into_Mn} and~\ref{prop:Permutation_into_Mn} that determining the existence of such a diagonal rescaling or permutation, respectively, admits polynomial time algorithms.

Towards our second goal, Theorem \ref{thm:OrderedEntries_Intro} provides a new class of matrices for which the SPN relaxation is tight. In particular, we show in Proposition \ref{prop:SeparableDNNTight} that this relaxation is tight for quadratic programs with \emph{separable} objective functions, answering \cite[Conjecture 1]{dey_convexification_2025-1}. Additionally,
we provide simple proofs of exactness of the relaxation for matrices with minimum entry on the diagonal and matrices which define convex objectives over the simplex, originally proven in  \cite{gokmen2022standard}.

Our results also admit a natural interpretation in the context of existing exactness results for SPN relaxations. A well-known connection links standard quadratic programming to the maximum weighted clique problem: for certain objective functions, standard quadratic programming and its SPN relaxation agree, computing the size of the maximum weighted clique of a perfect graph \cite{gibbons1997continuous, motzkin_maxima_1965, gokmen2022standard}. The conditions on such objective functions, which include relations between diagonal and off-diagonal entries of the objective matrix, imply that the \emph{sign graphs} of optimal solution matrices are perfect graphs, where a sign graph is a graph on $\{1,2,\dots,n\}$ whose edges correspond to entries taking a specified sign.
We show in Proposition~\ref{prop:SignGraphThreshold} that matrices with ordered off-diagonal entries have sign graphs which are \emph{threshold graphs} \cite{chvatal_aggregation_1977}. A threshold graph is a perfect graph which can be obtained by starting with a single vertex and successively adding vertices which are either connected to every other vertex or which are connected to no other vertex. From this perspective, Theorem \ref{thm:OrderedEntries_Intro} extends existing exactness results for SPN relaxations by removing some of the hypotheses on the objective function, provided that the associated sign graph of the optimal solution is threshold.

The remainder of the paper is organized as follows. Section \ref{sec:Background} fixes notation and provides background results on copositive and SPN matrices. In Section \ref{sec:MainTheorem}, we develop results about classes of matrices for which copositivity implies the existence of an SPN decomposition. Section \ref{sec:Orbits} discusses matrices which are in the orbit of the set of matrices with ordered off-diagonal entries. In Section \ref{sec:Optimization}, we apply the results about copositive matrices to certify the tightness of SPN relaxations of standard quadratic programming problems. Finally, in Section \ref{sec:SignGraphs}, we discuss the sign pattern graphs of matrices which are in the orbit of the ordered off-diagonal matrices.  

\section{Notation and Background}\label{sec:Background}

In this section, we fix the notation and provide background results on copositive and SPN matrices. 

\subsection{Notation} 

We denote by $\cS^n$ the set of $n\times n$ symmetric matrices with real entries. Let $\Cop_n \subseteq \cS^n$ denote the cone of $n\times n$ copositive matrices and $\SPN_n\subseteq \cS^n$ denote the cone of $n\times n$ matrices which are the sum of a positive semidefinite matrix and an entrywise nonnegative matrix. If $A \in \SPN_n$ for some $n$, then we say $A$ is SPN or $A$ admits an SPN decomposition.

If $A \in \cS^n$ is an $n\times n$ symmetric matrix, we let $A(i)$ be the $(n-1)\times (n-1)$ submatrix obtained by deleting row $i$ and column $i$. For a vector $v \in \bbR^n$, the notation $v \geq 0$ means that $v$ is entrywise nonnegative. Similarly, for matrices $A$ and $B$, $A\geq 0$ denotes that $A$ has nonnegative entries and $A \geq B$ means that $A-B \geq 0$. If $A$ is positive semidefinite, we write $A \succeq 0$. We denote by $E$ the $n\times n$ matrix with all entries equal to one. Similarly, $e \in \bbR^n$ is the vector with all entries equal to one so that $E = ee^\top$.

If $n$ is a natural number, we denote $[n] = \{1,2,\ldots, n\}$. For $v \in \bbR^n$, the support $\supp v \subseteq [n]$ is the set $\supp v = \{i \in [n]\; \vert \; v_i \neq 0\}$.

\subsection{Helpful Lemmas}

In this section, we collect background results about copositive matrices. 
\begin{lemma}[{\cite[Lemma 2]{diananda1962non}}]\label{lem:diag_elts}
If $A \in \Cop_n$ and $a_{ii} = 0$ for some $i \in [n]$, then $a_{i,j} \geq 0$ for all $j \in [n]$.
\end{lemma}

\begin{lemma}[{\cite[Theorem 4]{ping_criteria_1993}}]\label{lem:NegativePSD}
If $A \in \Cop_n$ and all off-diagonal elements of $A$ are nonpositive, then $A$ is positive semidefinite.
\end{lemma}

Matrices with nonpositive off-diagonal entries are studied extensively in \cite{berman_6_1994}. We recall some results here. In what follows, we frequently impose the conditions that a matrix has all nonpositive off diagonals or that a matrix is of the form $sI - B$ for some $s\in \bbR$ and $B \geq 0$. These are called $Z$-matrices and $M$-matrices, respectively in \cite{berman_6_1994,shaked-monderer_spn_2016}.

\begin{lemma}[{cf \cite[Lemma 2.8]{shaked-monderer_spn_2016}}]
If $A$ is a copositive matrix with all nonpositive off-diagonal entries, then $A = sI - B$ for some nonnegative matrix $B$ and $s \in \bbR$ such that $s$ is at least as large as the largest eigenvalue of $B$. 
\end{lemma}

\begin{lemma}\label{lem:SingularCopositiveNonnegativeKernel}
If $A$ is a singular copositive matrix will all nonpositive off-diagonal entries, then there is a vector $v \in \ker(A)$ such that $v \geq 0$.
\end{lemma}

\begin{proof}
Under the hypotheses of the lemma, $A$ is positive semidefinite with nontrivial kernel. So, the smallest eigenvalue of $A$ is 0 and therefore a unit vector $v \in \ker(A)$ is a minimizer of 

\[\min_{w\in \bbR^n} w^\top A w = \sum_{i = 1}^{n} a_{i,i}w_i^2 + 2\sum_{1\leq i < j \leq n}a_{i,j}w_iw_j \text{ s.t. } \|w\| = 1.\]
If $w$ is any unit vector and $\hat{w} = \begin{bmatrix} |w_1| & |w_2| & \ldots & |w_n|\end{bmatrix}^\top$ is the vector obtained by taking entrywise absolute values, then since $a_{i,j}\leq 0$ for all $i,j$ we have 

\[w^\top A w \geq \hat{w}^\top A \hat{w}.\]
So, there is a minimizer with nonnegative entries. 
\end{proof}

\begin{lemma}\label{lem:PSDSingularBlocks}
If $A$ is positive semidefinite with all nonpositive entries on the off-diagonal, then there is a permutation $P$ such that $P^\top A P = \begin{bmatrix} A_1 &0 \\ 0 & A_2\end{bmatrix}$, where $A_1$ is singular and $A_2$ is positive definite. Here we allow $A_1 = A$ or $A_2 = A$. 
\end{lemma}

\begin{proof}
If $A$ is positive definite, then there is nothing to show. Otherwise, let $v \in \ker(A)$ have $v \geq 0$. Let $v$ be a support inclusion maximal vector with this property: if $w \in \ker(A)$ has $\supp v \subseteq \supp w$ then $\supp v = \supp w$.  

If all entries of $v$ are nonzero, then $A_1 = A$ and we are done. Otherwise, let $P$ be a permutation such that $Pv = \begin{bmatrix} \hat{v}^\top & 0^\top \end{bmatrix}$ for some $\hat{v}$ with all nonzero entries. Set $A_1$ to be the submatrix of $A$ consisting of the rows and columns indexed by $\supp v$. Then, $P^\top A P = \begin{bmatrix} A_1 & A_{12}\\ A_{12}^\top & A_2\end{bmatrix}$ and we need to show that $A_{12} = 0$ and $A_{2}$ is nonsingular. Since $A$ is positive semidefinite (and therefore so is $P^\top A P$), we have that 

\[\begin{bmatrix} A_1 & A_{12}\\ A_{12}^\top & A_2\end{bmatrix}\begin{bmatrix} \hat{v} \\ 0 \end{bmatrix} = \begin{bmatrix} A_1 \hat{v}\\ A_{12}\hat{v}\end{bmatrix} = 0.\]
Since $A_{12}$ has all nonpositive entries and since $\hat{v}$ has all positive entries, this implies that $A_{12} = 0$. To conclude, note that if $A_2$ were not positive definite, then there would be $w \in \ker(A_2)$. But then, $\begin{bmatrix}\hat{v} & w \end{bmatrix} \in \ker(P^\top A P)$, contradicting the assumption that $v$ had inclusion maximal support. 
\end{proof}

\begin{lemma}[{cf \cite[Condition $N_{38}$]{berman_6_1994}, \cite[Lemma 2.9]{shaked-monderer_spn_2016}}]\label{lem:MmatrixInverses}
If $A$ is a nonsingular copositive matrix with nonpositive off-diagonal entries such that $A = sI - B$ for some $s\in \bbR, B\geq 0$ then $A^{-1}\geq 0$. If $B\geq A$ also has nonpositive off-diagonal entries, then $A^{-1} \geq B^{-1}$.
\end{lemma}

The following lemma deals with block matrices whose off-diagonal blocks have constant signs. It is a critical component of what follows, so we include a proof. 

\begin{lemma}[{\cite[Lemma 3.4]{shaked-monderer_spn_2016}}]\label{lem:row_signs}
Suppose that 

\[A = \begin{bmatrix} M & F \\ F^\top & B\end{bmatrix}, \quad M \in \bbR^{r\times r}, F \in \bbR^{r \times (n-r)}, B \in \bbR^{(n-r)\times (n-r)}.\]

\begin{itemize}
\item If $F \geq 0$ (entrywise), then $A \in \Cop_n \iff M \in \Cop_r \text{ and } B \in \Cop_{n-r}$,  and $A \in \mathcal{SPN}_n \iff M \in \SPN_r \text{ and } B \in \mathcal{SPN}_{n-r}$.
\item If $F \leq 0$ and $M$ is an invertible copositive matrix with negative off-diagonal entries, then $A \in \Cop_{n} \iff A/M \in \Cop_{n-r}$ and $A \in \SPN_n \iff A/M \in \SPN_{n-r}$. Here $A/M = B- F^\top M^{-1}F$ is the Schur Complement.
\end{itemize}
\end{lemma}

\begin{proof}
For the first claim, note that if $A \in \Cop_n$, then $M \in \Cop_r$ and $B \in \Cop_{(n-r)}$. Similarly, if $A \in \SPN_n$, then $M \in \SPN_r$ and $B \in \SPN_{(n-r)}$ since the principal submatrices of copositive matrices are copositive and similar for SPN matrices. For the other direction, note that since $F$ has nonnegative entries, $A$ is copositive (SPN) if both $M$ and $B$ are. 

For the second claim, note that we have the decomposition

\[A = \begin{bmatrix} M & F\\ F & F^\top M^{-1} F \end{bmatrix} + \begin{bmatrix}0 & 0\\ 0 & A/M \end{bmatrix}.\]
Since the first matrix in the sum is positive semidefinite, we have that if $A/M$ is copositive (SPN), then so is $A$. 

To see that copositivity of $A$ implies copositivity of $A/M$, note that for $v \in \bbR^{n-r}$ with $v \geq 0$, we have that $-M^{-1}Fv \geq 0$ by Lemma \ref{lem:MmatrixInverses} and 

\[\begin{bmatrix} -(M^{-1}Fv)^\top & v^\top \end{bmatrix}^\top \begin{bmatrix} M & F\\ F^\top &B\end{bmatrix}\begin{bmatrix} -(M^{-1}Fv) \\ v \end{bmatrix} = v^\top (A/M) v.\]

It remains to show that if $A$ is SPN, then so is $A/M$. If $A$ is SPN, then there is a positive semidefinite matrix $P$ with $P \leq A$. Permuting the rows and columns and applying Lemma \ref{lem:PSDSingularBlocks} if necessary, we have that 

\[P = \begin{bmatrix} P_{11} & 0 & P_{13}\\ 0 & P_{22} & P_{23}\\ P_{13}^\top & P_{23}^\top & P_{33} \end{bmatrix}.\]
Now, if $v\geq 0$ is a zero of $P_{11}$, then $\begin{bmatrix}v^\top & 0^\top&0^\top \end{bmatrix}^\top$ is a zero of $P$. Since $P\leq A$ and since $F \leq 0$, this implies that $P_{13} = 0$. Since $M$ was assumed to have nonpositive off-diagonal entries, it follows that the block structure of $A$ is given by 

\[A = \begin{bmatrix} M_1 & 0 & 0\\ 0 & M_2 & F_2\\ 0 & F_2^\top & B\end{bmatrix}.\]

So, $\displaystyle{A/M = \begin{bmatrix}0 & 0\\ 0 & B - F_2^\top M_2^{-1} F_2\end{bmatrix}}$. Now, $P_{22}^{-1} \geq M_2^{-1} \geq 0$ by Lemma \ref{lem:MmatrixInverses} since $P_{22} \leq M_2$ and both have nonpositive off-diagonal entries. It follows that

\[A/M = B - F_2^\top M_2^{-1}F_2 \geq P_{33} - P_{23}^\top P_{22}^{-1}P_{23} \succeq 0,\]
where the final matrix is positive semidefinite since it is a Schur complement of a positive semidefinite matrix. 
\end{proof}

\begin{lemma}[{\cite[Lemma 3.1]{shaked-monderer_spn_2016}}]\label{lem:COP=SPNwith_PSD_submatrix}
Let $A\in \Cop_n$ have a positive semidefinite submatrix of order $n-1$. Then, $A\in \SPN_n$.
\end{lemma}

\section{Main Theorem}\label{sec:MainTheorem}

In this section, we prove our main theorem, which states that copositive matrices with ordered rows and columns admit an SPN decomposition. We start in Section \ref{sec:MainTheoremGeneral} by describing a framework of sufficient conditions on a class of matrices for copositivity to be equivalent to the existence of an SPN decomposition. In Section \ref{sec:MainTheoremOrdered}, we specialize to the case of ordered rows and columns. Finally, in Section \ref{sec:MainTheoremBlock}, we describe block-matrix extensions of our results. 

\subsection{General Classes}\label{sec:MainTheoremGeneral}
In this section, we describe a class of matrices for which copositivity implies the existence of an SPN decomposition. 

\begin{theorem}\label{thm:MainTheoremGeneral}
Let $\cA = \bigsqcup_{n \geq 1}\cA_n$ be a disjoint union of sets of matrices $\cA_n \subseteq \cS^n$ such that for each $n \geq 5$,

\begin{enumerate}[label = (\ref{thm:MainTheoremGeneral}.\Roman*), ref = (\ref{thm:MainTheoremGeneral}.\Roman*)]
\item\label{itm:GenTheoremUniformRow} If $A \in \cA_n$, then there is a row $i$ such that all off-diagonal entries in row $i$ are nonnegative or all off-diagonal entries in row $i$ are nonpositive and at least one is nonzero.
\item\label{itm:GenTheoremDeleteNonneg} If $A \in \cA_n$, and row $i$ has all nonnegative off-diagonal entries, $A(i) \in \cA_{n-1}$.
\item\label{itm:GenTheoremSchur} If $A \in \cA_n$ has no rows with all nonnegative entries, then there is a row $i$ with all nonpositive off-diagonal entries and a nonzero diagonal entry such that $A/a_{i,i} \in \cA_{n-1}$. 
\end{enumerate}

If $A \in \cA$ is copositive, then it admits an SPN decomposition. 
\end{theorem}

\begin{proof}
We induct on $n$. Recall that for $n \leq 4$, every copositive matrix admits an SPN decomposition. 

Our base case is $n = 5$. Let $A \in \cA_5 \cap \Cop_5$. If there is a row $i$ such that all off-diagonal entries of row $i$ are nonnegative, then $A(i) \in \cS^{4}$ is copositive and therefore $A(i) \in \SPN_4$. By Lemma \ref{lem:row_signs}, this implies that $A \in \SPN_5$. Otherwise, there is a row $i$ such that all off-diagonal entries of row $i$ are nonpositive and at least one entry is negative. Since $A$ is copositive, this implies that $a_{ii}>0$ by Lemma \ref{lem:diag_elts}. By Lemma \ref{lem:row_signs}, we have that since $A$ is copositive, $A/a_{ii} \in \Cop_4 = \SPN_4$. By Lemma \ref{lem:row_signs} again, this implies that $A \in \SPN_5$. 

Now suppose that $n \geq 6$ and for all $k \leq n$, we have that $\cA_k \cap \Cop_k = \cA_k \cap \SPN_k$. Let $A \in \cA_n$.  If there is a row $i$ such that all off-diagonal entries of row $i$ are nonnegative, then $A(i) \in \cA_{n-1}$ is copositive and therefore $A(i) \in \SPN_{n-1}$. By Lemma \ref{lem:row_signs}, this implies that $A \in \SPN_n$. Otherwise, there is a row $i$ such that all off-diagonal entries of row $i$ are nonpositive and at least one entry is negative. As above, this implies that $a_{ii}>0$. By Lemma \ref{lem:row_signs}, we have that since $A$ is copositive, $A/a_{ii} \in \cA_{n-1}\cap \Cop_{n-1} = \cA_{n-1} \cap \SPN_{n-1}$. By Lemma \ref{lem:row_signs} again, this implies that $A \in \SPN_{n}$. 
\end{proof}

\begin{remark}
Two immediate examples of sets of matrices satisfying Conditions \ref{itm:GenTheoremUniformRow}, \ref{itm:GenTheoremDeleteNonneg}, and \ref{itm:GenTheoremSchur} are the set of nonnegative matrices and the set of matrices with nonpositive off-diagonal elements.
\end{remark}

\begin{example}[All positive off-diagonal entries contained in the last row]

Let $\cA_n\subseteq \cS^n$ be the set of matrices $A = \begin{bmatrix} A_1 & v\\ v^\top & \alpha\end{bmatrix}$, where $A_1 \in \cS^{n-1}$ has nonpositive off-diagonal entries, $v \in \bbR^{n-1},$ and $\alpha \in \bbR$. Note that if $A \in \cA_n$ is copositive, then $A_1$ is copositive. By Lemma \ref{lem:NegativePSD}, $A_1$ is PSD and therefore by Lemma \ref{lem:COP=SPNwith_PSD_submatrix} $A$ is SPN. 

We show that the set $\cA = \bigsqcup_{n \geq 1}\cA_n$ satisfies Conditions \ref{itm:GenTheoremUniformRow}, \ref{itm:GenTheoremDeleteNonneg}, and \ref{itm:GenTheoremSchur}. Let $A \in \cA_n$. If $v \geq 0$, then there is a row with all nonnegative elements. Otherwise, there is an index $i$ such that $v_i < 0$ and therefore the $i^{th}$ row of $A$ has all nonpositive off-diagonal elements. Next, the only possibly nonnegative row of $A$ is the final row $n$. $A(n)$ has only nonpositive off-diagonal entries, and therefore $A(n)\in \cA_{n-1}$. Finally, if $v$ is not nonnegative, then there is an index $i$ such that $v_i < 0$. Let $S = A/a_{ii}$ be the Schur complement, with rows indexed by $2,3,\ldots, n$ for convenience. Note that $s_{i,j} \leq a_{i,j}$ for all $1\leq i,j \leq n$. In particular, if $S = \begin{bmatrix} S_1 & u\\ u^\top & \beta \end{bmatrix}$ where $S_1 \in \cS^{n-2}$, it must be the case that $S_1$ has nonpositive off-diagonal entries and therefore $S \in \cA_{n-1}$. 
\end{example}

We record some immediate facts about sets of matrices satisfying Conditions \ref{itm:GenTheoremUniformRow},  \ref{itm:GenTheoremDeleteNonneg}, and \ref{itm:GenTheoremSchur}.

\begin{proposition}\label{prop:CombineGeneralClasses}
Suppose that $\cA^{(1)},\cA^{(2)},\ldots, \cA^{(k)}$ are sets of matrices satisfying Conditions \ref{itm:GenTheoremUniformRow},  \ref{itm:GenTheoremDeleteNonneg}, and \ref{itm:GenTheoremSchur}. The following sets of matrices also satisfy the conditions:

\begin{enumerate}
\item Block diagonal matrices with blocks in the sets $\cA^{(i)}$,
\item $\bigcup_{i = 1}^{k} \cA^{(i)}$,
\item The class $\cB$, where $A \in \cB$ if and only if $A \in \cA^{(1)}$ or $A$ has rows with nonpositive off-diagonal entries which, when removed, yield a matrix in $\cA^{(1)}$.
\end{enumerate}
\end{proposition}

The hypotheses on classes of matrices in Condition \ref{itm:GenTheoremSchur} is recursively defined. In particular, it is difficult to check that a class $\cA$ satisfies these hypotheses. Furthermore, given such a class $\cA$, it may be hard to check membership $A \in \cA$. In what follows, we work with classes $\cA$ which have concrete descriptions in terms of the ordering of matrix entries.  

\subsection{Ordered Off-Diagonal Entries}\label{sec:MainTheoremOrdered}  We now turn to a tractable case of Theorem \ref{thm:MainTheoremGeneral}, in which the off-diagonal entries of a matrix are ordered in rows and columns. For ease of notation, set 

\[\mathcal{M}_n = \{A = (a_{i,j}) \in \cS^n \; \vert \; a_{i,j}\leq a_{k,\ell} \text{ whenever } i\leq k \text{ and } j\leq \ell, \text{ and } i\neq j, k\neq \ell\}.\] 
Our main result is the following theorem:
\begin{theorem}\label{thm:Mn_SPN}
Let $A\in\mathcal{M}_n$. Then $A\in \Cop_n $ if and only if $A \in \mathcal{SPN}_n$.
\end{theorem}

Before proving the theorem, we record the following lemma, which allows us to study the copositivity of matrices in $\mathcal{M}_n$ using Lemma \ref{lem:row_signs}. 

\begin{lemma}\label{lem:Uniform_Row}
Suppose that $A\in \mathcal{M}_n$. Then, either 

\begin{enumerate}
\item $a_{n,j}\geq 0$ for all $1 \leq j \leq n-1$, or 
\item $a_{1,j} < 0$ for all $2\leq j \leq n$. 
\end{enumerate}
\end{lemma}

\begin{proof}
If $a_{n,1} \geq 0$, then $a_{n,j}\geq a_{n,1}\geq 0$ for all $1\leq j\leq n-1$ since $A \in \mathcal{M}_n$. Otherwise, $a_{1,n} = a_{n,1} < 0$. Once again, since $A \in \cM_n$, we have $a_{1,j}\leq a_{1,n} < 0$ for all $2 \leq j \leq n$. 
\end{proof}

\begin{remark}\label{rem:Uniform_Row_Alternate}
The similar statement that $A \in \cM_n$ has all nonpositive off-diagonal entries in its first row or strictly positive off-diagonal entries in its last row is also true. The proof is similar to that of Lemma \ref{lem:Uniform_Row}, but the cases are differentiated by $a_{n,1} > 0$ and $a_{n,1}\leq 0$. 
\end{remark}

We are now ready to prove Theorem~\ref{thm:Mn_SPN}.

\begin{proof}[Proof of Theorem~\ref{thm:Mn_SPN}]
For $n \leq 4$, we have that $\Cop_n = \SPN_n$ so that the claim holds. 

Let $\cM = \bigsqcup_{n\geq 1} \cM_n$. We show that $\cM$ satisfies the hypotheses of Theorem \ref{thm:MainTheoremGeneral}. Let $A \in \cM_n$.  Lemma \ref{lem:Uniform_Row} implies that $A$ has a row which either has all negative off-diagonal entries or all nonnegative off-diagonal entries and therefore $\cM_n$ satisfies \ref{itm:GenTheoremUniformRow}. 
For any $1\leq i\leq n$ we have $A(i) \in \cM_{n-1}$ since deleting a row and column from the matrix will not cause entries to become unordered, so $\cM_n$ satisfies \ref{itm:GenTheoremDeleteNonneg}.

For condition \ref{itm:GenTheoremSchur}, note that if there is no row with all nonnegative off-diagonal entries, then $a_{1,j} < 0 $ for all $2 \leq j \leq n$. The Schur complement of $A$ with respect to $a_{11}$ is

\[S:= A/a_{1,1} = A(1) - \frac{1}{a_{1,1}}\begin{bmatrix} a_{1,2} \\ \vdots \\ a_{1,n}\end{bmatrix}\begin{bmatrix} a_{1,2} & \dots & a_{1,n}\end{bmatrix} \in \Cop_{n-1}.\]
For convenience, we index the rows and columns of $S$ by $2,3, \dots, n$. Fix $i,j,k \in \{2,3,\ldots, n\}$ with $k > j$ and $i\neq j,k$. The entries of $S$ are given by 

\[s_{i,j} = a_{i,j} - \frac{1}{a_{1,1}}a_{1,i}a_{1,j}.\]
It follows that 

\[\begin{aligned}
s_{i,k} - s_{i,j} &= a_{i,k} - \frac{1}{a_{1,1}}a_{1,i}a_{1,k} - a_{i,j} + \frac{1}{a_{1,1}}a_{1,i}a_{1,j}\\
&= \left(a_{i,k} - a_{i,j}\right) + \frac{a_{1,i}}{a_{1,1}}(a_{1,j} - a_{1,k})\\
&\geq 0,\\
\end{aligned}\]
where the last inequality follows since $a_{i,k}\geq a_{i,j}$ and $a_{1,i} < 0$. So, $S \in \cM_{n-1}$. The result then follows from Theorem \ref{thm:MainTheoremGeneral} applied to the class $\cM$. \end{proof}

\subsection{Relaxing Ordering Assumptions}\label{sec:MainTheoremBlock}

Here, we extend the result of Theorem \ref{thm:Mn_SPN} to matrices with relaxed ordering constraints. We then use this to develop a block matrix generalization of Theorem \ref{thm:Mn_SPN}. The relaxed ordering depends on the signs of the entries in a row of the matrix. 

\begin{definition}[Positive Index of a Matrix]\label{def:idx}
Let $A = (a_{i,j}) \in \cS^n$. We define the positive index of $A$, denoted $\idx(A)$, to be the first row in which there is a positive off-diagonal entry. That is,  

\[\idx(A) = \min\{i\; \vert \; a_{i,j} > 0\text{ for some } j \in [n], j\neq i\}.\]

If all off-diagonal entries of $A$ are nonpositive, we set $\idx(A) = n+1$. 
\end{definition}

\begin{theorem}[Relaxed Ordering Assumption]\label{thm:RelaxedOrdering}
For $n\geq 1$, let $\cR_n \subseteq \cS^n$ be the set of matrices $M$ satisfying the following two conditions: 

\begin{enumerate}[label = (\ref{thm:RelaxedOrdering}.\Roman*), ref = (\ref{thm:RelaxedOrdering}.\Roman*)]
    \item\label{itm:RelaxedPositive} If $m_{i,j} > 0$, then $m_{i,j}\leq m_{i+1,j}$ and $m_{i,j}\leq m_{i,j+1}$ whenever these entries are defined
    \item\label{itm:RelaxedRowOrdered} If $k = \idx(M)$, then 

    \[m_{i,j} \leq m_{i,j+1} \text{ for all $i \in [n-1]$ and all } \max\{k,i+1\} \leq j \leq n-1.\]
\end{enumerate}

Then, $\cR_n \cap \Cop_n = \cR_n \cap \SPN_n$.
\end{theorem}

The ordering assumptions are illustrated in Figure \ref{fig:RelaxedOrderStructure}.

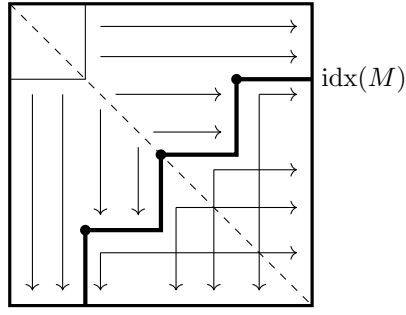
\begin{figure}
\begin{center}
\begin{tikzpicture}
    \draw[very thick] (0,0) rectangle (4,4);
    \draw (0,3) rectangle (1,4);
    \draw[ultra thick] (1,0) --(1,1)--(2,1)--(2,2)--(3,2)--(3,3)--(4,3);

    \draw[dashed] (0,4)--(4,0);

    \fill (1,1) circle (2pt);
    \fill (2,2) circle (2pt);
    \fill (3,3) circle (2pt);

    \draw[->] (1.2,3.7) -- (3.8,3.7);
    \draw[->] (1.2,3.3) -- (3.8,3.3);
    
    \draw[->] (0.3,2.8)--(0.3,0.2);
    \draw[->] (0.7,2.8) -- (0.7, 0.2);
    
    \draw[->] (1.4,2.8)--(2.8,2.8);
    \draw[->] (1.9,2.3)--(2.8,2.3);

    \draw[->] (1.2,2.6)--(1.2,1.2);
    \draw[->] (1.7,2.1)--(1.7,1.2);

    \draw[->] (1.2, 0.7)--(3.8,0.7);
    \draw[->] (1.2, 0.7)--(1.2,0.2);

    \draw[->] (2.2,1.3)--(3.8,1.3);
    \draw[->] (2.2,1.3)--(2.2,0.2);

    \draw[->] (2.7,1.8)--(3.8,1.8);
    \draw[->] (2.7,1.8)--(2.7,0.2);
    
    \draw[->] (3.3, 2.8) -- (3.8, 2.8);
    \draw[->] (3.3,2.8) -- (3.3,0.2);

    \node[right] at (4,3) {$\idx(M)$};  
\end{tikzpicture}
\caption{The ordering of off-diagonal elements in Theorem \ref{thm:RelaxedOrdering}. Off-diagonal entries are required to be nondecreasing in the directions of the arrows. The top left block has only nonpositive off-diagonal entries and no restrictions on the ordering of elements. The bottom right entries beneath the bold staircase are required to be positive.}\label{fig:RelaxedOrderStructure}
\end{center}
\end{figure}

\begin{proof}
We apply Theorem \ref{sec:MainTheoremGeneral} to the class $\cR = \bigsqcup_{n \geq 1}\cR_n$. 

We first show that if $M = (m_{i,j}) \in \cR_n$, then $M$ either has all nonpositive off-diagonal entries in its first row or all nonnegative off-diagonal entries in its last row. If the first row has all nonpositive off-diagonal entries, then we are done. Otherwise, $m_{1,j}>0$ for some $j > 1$. It follows that $m_{1,n} \geq m_{1,j}$ by Condition \ref{itm:RelaxedPositive}. We then have that $0\leq m_{1,n} \leq m_{2,n}\leq \dots \leq m_{n-1,n}$ by Condition \ref{itm:RelaxedPositive} again. So, the last row has all nonnegative off-diagonal entries. So, Condition \ref{itm:GenTheoremUniformRow} is satisfied. 

Since deleting a row and column with nonnegative entries does not affect ordering or signs of the entries in other rows, if $M \in \cR_n$ and row $i$ has all nonnegative entries, then $M(i) \in \cR_{n-1}$. So, Condition \ref{itm:GenTheoremDeleteNonneg} is satisfied. 

Finally, we need to show that if there are no rows with all nonnegative entries, then there is a row $i$ such that the Schur complement $M/m_{i,i} \in \cR_{n-1}$. We show this for $i = 1$. Partition $M$ as 

\[M = \begin{bmatrix} \alpha & v^\top\\ v & M_1\end{bmatrix},\]
where $v = \begin{bmatrix} m_{1,2} & m_{1,3} & \dots & m_{1,n}\end{bmatrix}^\top\leq 0$. We are interested in $S = M/m_{1,1} = M_1 - \frac{1}{\alpha} vv^\top$. The off-diagonal entries of $S$ are 

\[s_{i,j} = m_{i+1,j+1} - \frac{m_{1,i+1}m_{1,j+1}}{\alpha}.\]
Note that $s_{i,j}\leq m_{i+1,j+1}$. 

We show that $S$ satisfies Condition \ref{itm:RelaxedPositive}. Suppose that $s_{i,j} > 0$. Then, 

\[\begin{aligned}
    s_{i,j+1} - s_{i,j} &= \left(m_{i+1,j+2} - m_{i+1,j+1}\right) - \left(\frac{m_{1,i+1}m_{1,j+2}}{\alpha} - \frac{m_{1,i+1}m_{1,j+1}}{\alpha}\right)\\
    &= \left(m_{i+1,j+2} - m_{i+1,j+1}\right) - \frac{m_{1,i+1}}{\alpha}\left(m_{1,j+2} - m_{1,j+1}\right).
\end{aligned}\]
Since $s_{i,j} > 0$, we have that $m_{i+1,j+1} > 0$ and therefore $k = \idx(M) \leq i+1$. So, $m_{i+1,j+2} - m_{i+1,j+1} \geq 0$ and $m_{1,j+2} - m_{1,j+1} \geq 0$ by Condition \ref{itm:RelaxedRowOrdered} applied to $M$. So, $s_{i,j+1} -s_{i,j} \geq 0$.

Similarly, 

\[
\begin{aligned}
s_{i+1,j} - s_{i,j} &= \left(m_{i+2,j+1} - m_{i+1,j+1}\right) - \left(\frac{m_{1,i+2}m_{1,j+1}}{\alpha} - \frac{m_{1,i+1}m_{1,j+1}}{\alpha}\right)\\
    &= \left(m_{i+2,j+1} - m_{i+1,j+1}\right) - \frac{m_{1,j+1}}{\alpha}\left(m_{1,i+2} - m_{1,j+1}\right).
\end{aligned}
\]
As above, $k \leq i+1$ so that $m_{i+2,j+1}-m_{i+1,j+1} \geq 0$ and $m_{1,i+2} - m_{1,i+1}\geq 0$ by Condition \ref{itm:RelaxedRowOrdered} applied to $M$. So, $S$ satisfies Condition \ref{itm:RelaxedPositive}.

Finally, we show that $S$ satisfies Condition \ref{itm:RelaxedRowOrdered}. Let $k_S = \idx(S)$. We need to show that for any $i \in [n-1]$, we have $s_{i,j}\leq s_{i,j+1}$ for $\max\{k_S,i+1\}\leq j \leq n-1$.  Now, 

\[s_{i,j+1} - s_{i,j} = \left(m_{i+1,j+2} - m_{i+1,j+1}\right) - \frac{m_{1,i+1}}{\alpha}\left(m_{1,j+2} - m_{1,j+1}\right).\]

Since $m_{i+1,j+1} \geq s_{i,j}$, we have $k_S \geq k-1$. Therefore, $j \geq k_S$ implies that $j + 1 \geq k$. By Condition \ref{itm:RelaxedRowOrdered} applied to $M$, we have $m_{i+1,j+2} \geq m_{i+1,j+1}$ and $m_{1,j+2} \geq m_{1,j+1}$. So, $s_{i,j+1} - s_{i,j} \geq 0$ and $S$ satisfies Condition \ref{itm:RelaxedRowOrdered}. 

Therefore, $\cR$ satisfies hypotheses \ref{itm:GenTheoremUniformRow}, \ref{itm:GenTheoremDeleteNonneg}, and \ref{itm:GenTheoremSchur}. So, $\cR_n \cap \Cop_n = \cR_n \cap \SPN_n$ for all $n$ by Theorem \ref{thm:MainTheoremGeneral}.
\end{proof} 

Theorem \ref{thm:RelaxedOrdering} has an immediate corollary for block matrices. 

\begin{corollary}    
[Block Sign-pattern]\label{ex:BlockSign_SPN}
Let
\begin{equation}\label{eq:BlockStructure1} M=\begin{bmatrix}
A & B\\
B^\top & C
\end{bmatrix}\in \mathcal S^{k+n},
\end{equation}
where $A\in \mathcal{S}^k$, $B\in \mathbb{R}^{k\times n}$, and $C\in \mathcal{S}^n$. Assume that

\begin{enumerate}
\item $a_{i,j}\leq 0$ for all $i\neq j$,
\item $b_{i,\alpha}\leq 0$ for all $i,\alpha$,
\item for each fixed $i$, the row $(b_{i,1},\dots,b_{i,n})$ is nondecreasing, (not requiring this for $A$)
\item $C\in \mathcal M_n$.
\end{enumerate}
\end{corollary}
\begin{proof}
Such block matrices satisfy the ordering conditions \ref{itm:RelaxedPositive} and \ref{itm:RelaxedRowOrdered}. Indeed, any positive off-diagonal entry of $M$ is in the block $C$, so that $\idx(M) \geq k+1$. Since $C\in \cM_n$, it is row-ordered and column-ordered and therefore Condition \ref{itm:RelaxedPositive} holds. Similarly, Condition \ref{itm:RelaxedRowOrdered} holds since $B$ and $C$ are both row-ordered and $\idx(M) \geq k+1$.
\end{proof}

It follows from Proposition \ref{prop:CombineGeneralClasses} that if $A_1, A_2,\ldots, A_k$ are $n_i \times n_i$ symmetric matrices with $A_i \in \cM_{n_i}$ for each $i \in [k]$, then a block diagonal matrix with blocks $A_1,A_2,\ldots, A_k$ is copositive if and only if each $A_i$ admits an SPN decomposition. A slightly more general statement holds for matrices whose blocks only overlap in a single row and column. 

\begin{proposition}\label{prop:almost_block_Mn}
    Let $A \in \Cop_{n+k+1}$ be a copositive matrix with block structure

    \[A = \begin{bmatrix} A_1 & b_1 & 0\\ b_1^\top & c & b_2^\top\\ 0& b_2 & A_2 \end{bmatrix},\]
    where $A_1 \in \cS^n$, $b_1\in \bbR^n,$ $A_2 \in \cS^k,$ $b_2 \in \bbR^k$, and $c \in \bbR$. If the two submatrices $\begin{bmatrix}A_1 & b_1\\ b_1^\top & c \end{bmatrix}$ and $\begin{bmatrix} c & b_2^\top\\ b_2 & A_2\end{bmatrix}$ are in $\cM_{n+1}$ and $\cM_{k+1}$, respectively, then $A \in \SPN_{n+k+1}$.
\end{proposition}

\begin{proof}
By \cite[Lemma 3.10]{shaked-monderer_spn_2016}, there are matrices $B_1\in \Cop_{n+1}$ and $B_2\in \Cop_{k+1}$ and $0\leq \gamma\leq c$ such that

\[B_1 + \begin{bmatrix} 0& 0 \\0 &\gamma\end{bmatrix} = \begin{bmatrix}A_1 & b_1\\ b_1^\top & c \end{bmatrix}, \quad B_2 + \begin{bmatrix} c-\gamma & 0\\ 0 & 0 \end{bmatrix} = \begin{bmatrix} c & b_2^\top\\ b_2 & A_2\end{bmatrix}, \quad \text{ and } A = \begin{bmatrix} B_1 & 0\\ 0 & 0\end{bmatrix} + \begin{bmatrix} 0 & 0\\ 0 & B_2 \end{bmatrix}.\]
Since membership in $\cM_{n+1}$ and $\cM_{k+1}$ does not depend on the value of diagonal entries, it follows that $B_1 \in \cM_{n+1}$ and $B_2 \in \cM_{k+1}$. Since $B_1$ and $B_2$ are copositive by hypothesis, they are SPN. So, $A$ is the sum of two SPN matrices and is therefore also SPN.
\end{proof}

\section{Permutations and Rescalings of $\cM_n$}\label{sec:Orbits}

The convex cones $\Cop_n$ and $\SPN_n$ are closed under simultaneous permutation of rows and columns and positive diagonal rescalings. That is, if $A$ is copositive (SPN), then so is $P^\top A P$ for any permutation matrix $P$ and $D A D$ for any diagonal matrix $D$ with positive entries on the diagonal. 
So, it is natural to consider which matrices arise from applying such permutations and rescalings to copositive elements of $\cM_n$. 

\begin{definition}
Let $G_n$ be the group of matrices which are the product of a permutation matrix and a diagonal matrix with positive entries on the diagonal:  
\[
\begin{split}
G_n = \{PD \; \vert \; & P \text{ is an $n\times n$ permutation matrix and } \\ &D \text{ is an } n \times n \text{ diagonal matrix with positive diagonals}\}.
\end{split}
\]
Note that $G_n$ acts on $\cS^n$ via $g\cdot X = g^\top X g$. Set $G_n\cM_n = \{g\cdot X\;\vert g\in G_n\text{ and } X \in \cM_n\}$ to be the orbit of $\cM_n$ under the action of $G_n$ and define $\cK_n$ to be the conical hull of $G_n\cM_n \cap \Cop_n$. 

\[\cK_n = \cone\left(G_n\cM_n \cap \Cop_n\right).\]
\end{definition}

\begin{corollary}\label{cor:GeneratedConeSPN}
The cone $\cK_n$ is contained in $\SPN_n$.
\end{corollary}

\begin{proof}
Let $A \in \cK_n$. Then $A = \sum_{i = 1}^r \lambda_i D_i^\top P_i^\top A_i P_i D_i$, where each $\lambda_i \geq 0$, each $D_i$ is an $n\times n$ diagonal matrix with positive entries on the diagonal, each $P_i$ is an $n\times n$ permutation matrix, and each $A_i \in \cM_n$, and for each $i$, $D_i^\top P_i^\top A_i P_i D_i \in \Cop_n$. Since $\Cop_n$ is invariant under the action of $G_n$, it follows that for each $i$, 

\[(P_iD_i)^{-\top}(D_i^\top P_i^\top)A_i(P_iD_i)(P_iD_i)^{-1} = A_i \in \Cop_n.\]
Since $A_i \in \cM_n$, we have that $A_i \in \SPN_n$. But then, $(D_i^\top P_i^\top)A_i(P_iD_i) \in \SPN_n$ as well for each $i$. So, $A$ is a conical combination of elements of $\SPN_n$ and therefore $A \in \SPN_n$.
\end{proof}

\begin{remark}\label{rem:NonClosedGroup}
We note that the group $G_n$ is not closed since the diagonal matrix $D$ is required to have strictly positive entries. Elements of the boundary of the closure $\mathrm{cl}(G_n\cM_n)$ will have a row and column of zeros.
\end{remark}

We note that the closure of orbits $G_n\cM_n \cap \Cop_n$ does not cover the entirety of the nonnegative cone. 

\begin{example}[SPN matrix which is not an element of $G_n\cM_n$] 
Consider the matrix 

\[M = \begin{bmatrix} 1 & 1 & 0 & 0 & 1\\ 1 & 1 & 1 & 0 & 0\\ 0 & 1 & 1 & 1 & 0\\ 0 & 0 & 1 & 1 & 1\\ 1 & 0 & 0 & 1 & 1\\\end{bmatrix}.\]

The matrix $M$ is entrywise nonnegative and therefore $M \in \SPN_5$. The nonzero off-diagonal entries of $M$ form a $5$-cycle. So, for any $5\times 5$ permutation matrix $P$, the nonzero off-diagonal elements of $P^{\top}MP$ must also form a $5$-cycle. In particular, there must be a row $i \in \{2,\ldots, 5\}$  such that $(P^{\top}MP)_{i,1} = 1$. Since there are exactly 2 nonzero off-diagonal entries in row $i$, it follows that the off-diagonal entries of row $i$ cannot be nondecreasing. Since scaling by a positive diagonal matrix preserves zeros, it follows that the orbit of $M$ under $G_5$ does not intersect $\cM_5$. Moreover, $M$ has no zero rows, so $M$ is not in the closure of $G_5\cM_5$.
\end{example}

However, the conical hull of orbits contains the entirety of the nonnegative cone as well as rank one extreme rays of the positive semidefinite cone with specific sign patterns.

\begin{proposition}[Extreme Rays of $\SPN_n$ contained in $\cK_n$]\label{prop:ExtremeRaysKn}
The following extreme rays of $\SPN_n$ are contained in $\cK_n$:
\begin{enumerate}
    \item The nonnegative matrices $E_{i,j}$ with entries $(i,j)$ and $(j,i)$ equal to $1$ and all other entries zero.
    \item Positive semidefinite matrices $vv^\top$ where $v\in \bbR^n$ has no zero entries and either $v$ or $-v$ has at most one negative entry. 
    \item Positive semidefinite matrices $vv^\top$ where $v \in \bbR^n$ has exactly one positive entry, one negative entry, and $n - 2$ entries equal to zero. 
\end{enumerate}
\end{proposition}

\begin{proof}
\begin{enumerate} 
\item For any $i,j \in [n]$, there is a permutation which exchanges $i$ with $n-1$ and exchanges $j$ with $n$ and leaves all other entries fixed. If $P$ is the $n \times n$ matrix representing this permutation, then $P^\top E_{i,j}P = E_{n-1,n} \in \cM_n$.

\item We can assume that $v$ has at most one negative entry since $vv^\top  = (-v)(-v^\top)$. If all entries of $v$ have the same sign, then $vv^\top$ is nonnegative and the claim follows.  Otherwise, permuting the entries of $v$ if necessary, we can take $v_1 < 0$ and $v_i > 0$ for $i = 2, \ldots, n$. Then, by rescaling with the diagonal matrix

\[D = \begin{bmatrix} \frac{1}{|v_1|} & & &\\ & \frac{1}{v_2} & & \\ && \ddots & \\ & & & \frac{1}{v_n} \end{bmatrix},\]
it suffices to show that the matrix $\begin{bmatrix}1 & -e_{n-1}^\top \\ -e_{n-1} & e_{n-1}e_{n-1}^\top\end{bmatrix} \in \cM_{n}$, where $e_{n-1} \in \bbR^{n-1}$ has all entries equal to $1$. The off-diagonal entries of this matrix are nondecreasing in rows and columns, so the claim holds. 

\item Permuting $v$ if necessary, we can assume that $v_1 < 0$ and $v_2 > 0$. The result follows by scaling by the diagonal matrix $D$ with $d_{1,1} = \frac{1}{|v_1|}, d_{2,2} = \frac{1}{v_2}$, and $d_{i,i} = 1$ for $i \geq 3$, and the observation that 

\[\begin{bmatrix} 1 & -1 & 0_{1, n-2}\\ -1 & 1 & 0_{1,n-2}\\ 0_{n-2,1} &0_{n-2,1} & 0_{n-2.n-2}\end{bmatrix} \in \cM_n,\]
where $0_{\ell, k} \in \bbR^{\ell \times k}$ has all entries equal to 0. 
\end{enumerate}
\end{proof}

It follows from Proposition \ref{prop:ExtremeRaysKn} that $\cK_3 = \SPN_3 = \Cop_3$. Indeed, if $v \in \bbR^3$ is nonzero and has all nonzero entries then either $v$ or $-v$ has at most one negative entry. If all nonzero entries of $v$ have the same sign then $vv^\top$ has nonnegative entries. The only remaining case is when $v$ has one positive, one negative, and one zero entry. So, $\cK_3$ is a convex cone which contains all rank one positive semidefinite matrices and all nonnegative matrices and therefore $\cK_3 = \SPN_3 = \Cop_3$. However, the cone $\cK_n$ is strictly contained in $\SPN_4$ for $n\geq 4. $

\begin{proposition}[$\cK_n\neq \SPN_n$ for $n \geq 4$]
Let $n\ge 4$ and  
\[
v = (-1,-1,\underbrace{1,\dots,1}_{n-2})^\top .
\]
The positive semidefinite matrix $A = vv^\top$ is not an element of $\cK_n$. In particular, $\cK_n$ is a \emph{strict} subset of $\SPN_n$ for $n \geq 4$. 
\end{proposition}

\begin{proof}
Note that $\ker A$ contains $n -1$ linearly independent nonnegative vectors. So, $A$ is an extreme ray of the cone $\Cop_n$. It therefore suffices to show that $A \not \in G_n\cM_n$. By Lemma \ref{lem:Uniform_Row}, any matrix in $\cM_n$ must have a row with all nonnegative off-diagonal entries or a row with all negative off-diagonal entries. Any matrix in $G_n\cM_n$ must therefore also have a row with all nonnegative off-diagonal entries or a row with all negative off-diagonal entries. Since every row of $A$ has at least one positive and one negative off-diagonal entry, $A \not \in G_n\cM_n$.
\end{proof}

\subsection{Checking membership in $G_n\cM_n$}

Here, we discuss the problem of checking that a given matrix is in the orbit of $\cM_n$. The existence of a positive diagonal rescaling is a linear programming feasibility problem. 

\begin{proposition}\label{prop:Rescale_into_Mn}
Let $A = (a_{i,j}) \in \cS^n$. There is a diagonal matrix $D$ with positive diagonal entries such that $DAD \in \cM_n$ if and only if the following system of linear inequalities is feasible
\begin{equation}\label{eq:RescaleLP}\exists \, d_i \text{ s.t. } d_ka_{i,k} - d_ja_{i,j}\geq 0 \text{ for all } i,j,k \text{ with } j<k\text{ and } i\neq j, i\neq k \text{ and } d_i \geq 1\end{equation}

If \eqref{eq:RescaleLP} has a feasible solution $d$, then $D = \diag(d_1,d_2,\ldots, d_n)$ satisfies $DAD \in \cM_n$. 
\end{proposition}

\begin{proof}
Let $D = \diag(d_1,d_2,\ldots, d_n)$ with each $d_i$ positive and set $\tilde{A}  = DAD$. Note that the off-diagonal entries of $\tilde{A}$ are $\tilde{a}_{i,j} = d_id_j a_{i,j}$ for $i \neq j$. Now, if $\tilde{A} \in \cM_n$ if and only if $\tilde{a}_{i,j} \leq \tilde{a}_{i,k}$ for all $i,j,k$ with $j < k$ and $i\neq j, i\neq k$. But

\[\begin{aligned}
\tilde{a}_{i,j} \leq \tilde{a}_{i,k} &\iff d_id_ja_{i,j} \leq d_id_ka_{i,k}\\
&\iff d_i(d_ka_{i,k} - d_ja_{j,k}) \geq 0\\
&\iff d_ka_{i,k} - d_ja_{j,k} \geq 0
\end{aligned}\]
where the last equivalence is due to $d_i > 0$. Note that a solution $d>0$ exists if and only if a solution $d\geq 1$ exists by rescaling since the last inequality is homogeneous. 
\end{proof}

Similarly, the existence of a permutation bringing a matrix into $\cM_n$ can be determined efficiently. 

\begin{proposition}\label{prop:Permutation_into_Mn}
Let $A = (a_{i,j}) \in \cS^n$. There is a polynomial time algorithm to decide if there is a permutation $P$ such that $P^\top A P \in \cM_n$.   
\end{proposition}

\begin{proof}
    Since the matrices involved are symmetric, it can be verified that $B\in \mathcal{M}_n \iff
    b_{i,j} \leq b_{i,k}$ for all $j < k$ and $i \not \in \{j,k\}.$ That is, we only need to check every row of $B$.
    Now suppose $B =P^{\top}AP$, where $P$ is a permutation matrix and let $\pi$ be the permutation corresponding to $P$, so that  $b_{r,s} =a_{\pi(r),\pi(s)}$. Then by the observation before $B \in \mathcal{M}_n$ if and only if $$a_{i,j} < a_{i,k}\qquad\textup{whenever } \pi^{-1}(j)< \pi^{-1}(k),$$
    for all distinct $i,j,k\in [n]$. 

    Define a directed graph $D_A$ on the vertex set $[n]$ by inserting an arc
$j\to k$ whenever there exists $i\notin\{j,k\}$ such that $a_{i,j}<a_{i,k}$.
Then a permutation $\pi$ satisfies the displayed implication if and only if
$\pi$ is a topological ordering of $D_A$. Hence such a permutation exists if and only if $D_A$ is acyclic.

The graph $D_A$ can be constructed in $O(n^3)$ time by checking all triples $(i,j,k)$, and acyclicity can be tested in polynomial time by topological sorting.
\end{proof}

Propositions \ref{prop:Rescale_into_Mn} and \ref{prop:Permutation_into_Mn} state that there is a polynomial time algorithm to check if a given matrix is in the orbit of $\cM_n$ under the action of the symmetric group or under the action of the group of positive diagonal matrices. However, copositivity is preserved under action by the group $G_n$, which is a larger set. 

\section{Applications to Optimization}\label{sec:Optimization}

Consider the standard quadratic programming (St-QP) problem 
\begin{equation}\label{eq:non-sepStQP}
 z^*  :=  \min \{x^\top Q x \;\vert\; x \in\Delta^n\}, 
\end{equation} 
where $Q\in\cS^n$ and $\Delta^n =\{x\in\mathbb{R}^n_+ \;\vert\; e^\top x = 1\}$ is the standard simplex. 
\begin{remark}\label{rem:St-QP}
Note that a problem of the form $\min \{x^\top  \tilde Q x + 2\alpha^\top  x \;\vert\; x \in\Delta^n\}$ can also be put in the form of (St-QP) by setting $Q := \tilde Q + \alpha e^\top  + e\alpha^\top $.
\end{remark}

St-QP is known to be NP-Hard and has attracted significant attention in the conic programming literature. %
A classical result in~\cite{burer2009copositive} states that there is an equivalent conic reformulation of St-QP as the following completely positive program (CP):
\begin{equation*}\label{eq:non-sepStQP-CP}
  z^* =  z_{\textup{CP}} :=   \min \{  Q \bullet X  \;\vert\; E \bullet X = 1, \ X \in  \CP^n\}.
\end{equation*}
Here, $\CP^n$ denotes the cone of $n\times n$  completely positive matrices and for $A,B \in \cS^n,$ $A \bullet B = \trace(AB)$. The dual problem to CP is the following copositive program (COP):
\begin{equation}\label{eq:non-sepStQP-COP}
  z_{\textup{COP}} :=   \max \{  \lambda\;\vert\; Q- \lambda  E \in  \Cop_n\}.
\end{equation}
Since both CP and COP are intractable for $n\ge5$, their tractable approximations in the form of doubly nonnegative relaxation (DNN) and SPN are frequently used. Here, we define the DNN relaxation of St-QP as
\begin{equation*}\label{eq:non-sepStQP-DNN}
  z_{\textup{DNN}} :=   \min \{  Q \bullet X \;\vert\; E \bullet X = 1, \ X \in  \DNN_n\},
\end{equation*}
where $\DNN_n$ denotes the cone of $n\times n$ doubly nonnegative matrices. Finally, we have the following SPN formulation
\begin{equation}\label{eq:non-sepStQP-SPN}
  z_{\textup{SPN}} :=   \max \{  \lambda\;\vert\; Q- \lambda E  \in     \SPN_n\}. 
\end{equation}
Since the strong duality holds between CP and COP as well as between DNN and SPN, we have the following relationship:
\[
z_{\textup{SPN}} = z_{\textup{DNN}} \le z_{\textup{CP}} = z_{\textup{COP}}.
\]
Hence, we have
\[
z_{\textup{SPN}} = z_{\textup{COP}} \iff  z_{\textup{DNN}} = z_{\textup{CP}} ,
\]
which implies that the DNN relaxation is exact whenever SPN and COP yield the same objective function value. A critical observation in this regard is the following.
\begin{observation}\label{obs:mainObservation}
Let $\cQ\subseteq \cS^n$. Using (\ref{eq:non-sepStQP-COP}) and (\ref{eq:non-sepStQP-SPN}), we observe that if $\cQ \cap (\Cop_n + \lambda E)  =  \cQ \cap (\SPN_n + \lambda E) $ for $\lambda\in\mathbb{R}$, then $z_{\textup{SPN}} = z_{\textup{COP}}$, that is the DNN relaxation is exact. 
Moreover, if $\cQ + \lambda E =\cQ$ for $\lambda\in\mathbb{R}$, then the sufficient condition simply becomes  $\cQ \cap \Cop_n  =  \cQ \cap \SPN_n  $.
\end{observation}


\subsection{Prior Exactness Results}

In this section, we discuss previously known exactness results for DNN relaxations.
We present these results here for completeness and to provide context for the new exactness results in Section~\ref{sec:NewResults}. In particular, the proofs below illustrate how Observation~\ref{obs:mainObservation} unifies and simplifies the arguments underlying prior results.

First, we consider the case in which the objective matrix has its minimum entry on the diagonal. Let $\cQ_n^{\min} := \{A = (a_{ij})\in \cS^n \; \vert \; \min_{i,j} a_{i,j} = \min_{i} a_{ii}\}$ be the collection of $n\times n$ symmetric matrices with minimum entry on the diagonal. Notice that $\cQ_n^{\min} + \lambda E = \cQ_n^{\min}$ for $\lambda\in\mathbb{R}$.  The following result implies that the DNN relaxation of problem~\eqref{eq:non-sepStQP} is exact when $Q\in\cQ_n^{\min}$. 

\begin{proposition}[{\cite[Proposition 2]{gokmen2022standard}}]
 $Q \in \cQ_n^{\min}\cap\Cop_n $ if and only if $Q \in \cQ_n^{\min} \cap \mathcal{SPN}_n$.
\end{proposition}
\begin{proof}
 We only need to prove the forward direction. Let $Q \in  \cQ_n^{\min}\cap\Cop_n$. Since $Q\in \Cop_n$, the diagonal entries are all nonnegative, hence, $\min_i Q_{ii} \ge 0$. Since $Q\in \cQ_n^{\min}$, this also implies that $\min_{i,j} Q_{ij} \ge 0$, meaning that $Q$ is a nonnegative matrix. Hence, we deduce that $Q\in\SPN_n$.
\end{proof}
Note that a relevant subset of $\cQ_n^{\min}$ is the set of matrices that are negative semidefinite over the standard simplex (see, Proposition~3 in~\cite{gokmen2022standard}). That is,
\[
\cQ_n^- := \{ Q \in \cS^n \;\vert\; d^\top Q d \le 0 \ \text{ for all } d\in \mathbb{R}^n \text{ with } e^\top d = 0\} \subseteq \cQ_n^{\min}.
\]

Next, we consider the case of objective functions which are convex over the simplex. Let 
\[\cQ_n^+ := \{ Q \in \cS^n \; \vert \; d^\top  Q d \ge 0 \ \text{ for all } d\in \mathbb{R}^n \text{ with } e^\top  d = 0\}.\]

Notice that $\cQ_n^{+} + \lambda E =  \cQ_n^{+}$ for $\lambda\in\mathbb{R}$. The following result implies that the DNN relaxation of problem~\eqref{eq:non-sepStQP} is exact when $Q\in\cQ_n^+$. The result appears as \cite[Proposition 4]{gokmen2022standard} in which the authors provide an explicit SPN decomposition in their proof. Here we provide a simple alternative proof which does not require the construction of such a decomposition. 
\begin{proposition}
 $Q \in \cQ_n^+\cap\Cop_n $ if and only if $Q \in \cQ_n^+ \cap \mathcal{SPN}_n$.
\end{proposition}

\begin{proof}
    We only need to prove the forward direction. Let $Q \in  \cQ_n^+\cap\Cop_n$. Since $Q\in \cQ_n^+$, there exists $\lambda\ge0$ such that $Q+\lambda ee^\top  \in \PSD_n$ due to Finsler Lemma. If $\lambda=0$, then $Q\in \PSD_n \subseteq \SPN_n$. So, let us assume that $\lambda > 0$.

    Since $Q \in \Cop_n$ and $\lambda >0$, we have $\begin{bmatrix}
        Q+\lambda ee^\top  & - e \\ -e^\top  & \frac{1}{\lambda}
    \end{bmatrix} \in \Cop_{n+1}$ due to Lemma~\ref{lem:row_signs} (notice that the Schur complement with respect to $\begin{bmatrix}\frac{1}{\lambda}\end{bmatrix}$ is exactly $Q$). Observe also that since $ Q+\lambda ee^\top  \in \PSD_n $, we have that $\begin{bmatrix}
        Q+\lambda ee^\top  & - e \\ -e^\top  & \frac{1}{\lambda}
    \end{bmatrix} \in \SPN_{n+1}$ due to Lemma~\ref{lem:COP=SPNwith_PSD_submatrix}. Finally, we apply Lemma~\ref{lem:row_signs} again to conclude that  $Q \in \SPN_n$. 
\end{proof}

\subsection{New Exactness Results}\label{sec:NewResults}
We now leverage the results of Section~\ref{sec:MainTheorem} to prove the exactness of the DNN relaxation for new classes of objective functions. 

Our first result follows directly from Observation~\ref{obs:mainObservation} and Theorem~\ref{thm:Mn_SPN}.
\begin{corollary}\label{cor:Mn_DNN_tight}
If $Q \in \cM_n$, then $z^* = z_{\textup{SPN}}$ for the standard quadratic program with objective function $Q$. 
\end{corollary}

\begin{proof}
$\cM_n + \lambda E = \cM_n$ for all $\lambda \in \bbR$ since adding $\lambda$ to all entries of a matrix does not change their ordering. 

\end{proof}

We now apply Corollary~\ref{cor:Mn_DNN_tight} to an example from
 \cite{gokmen2022standard} in which
 the tightness of the DNN relaxation is proven for three classes of matrices.  The matrix $Q\in\mathcal{S}^5$ in the following example does \textit{not}  belong to any of these classes yet  the DNN relaxation is proven to be tight since the authors provide an explicit SPN certificate. We show that the matrix $Q$ indeed belongs to $\mathcal{M}_5$.
\begin{example}[{\cite[Example 4]{gokmen2022standard}}]\label{ex:Extraneous1}
Consider the matrix 

\[Q = \begin{bmatrix}2 & 2 & 2 & 2 & 2\\ 2 & 2 & 2 & 2 & 2\\  2&2&2&1&2\\ 2&2&1&2&0\\2&2&2&0&2\end{bmatrix}.\]
We have 

\[P^{\top}QP = \begin{bmatrix}2 & 0 & 1 & 2 & 2\\ 0 & 2 & 2 & 2 & 2\\ 1 & 2 & 2 & 2 & 2\\ 2 & 2 & 2 & 2 & 2\\ 2 & 2 & 2 & 2 & 2 \end{bmatrix} \in \cM_5,\]
where $P$ is the permutation matrix corresponding to $(1,2,3,4,5)\mapsto (4,5,3,1,2)$.

The optimal value of the program 

\[\max \lambda \text{ s.t. } Q - \lambda E \in \Cop_5\]
is $z_{\textup{COP}} = 1$.  Corollary \ref{cor:Mn_DNN_tight} implies that $z_{\textup{SPN}} = z_{\textup{COP}}$.
\end{example}

We next show that Corollary~\ref{cor:Mn_DNN_tight} resolves a conjecture of~\cite{dey_convexification_2025-1}. The conjecture concerns separable standard quadratic programs, which are problems of the form
\begin{equation}\label{eq:StandardSeparableProgram}
\min \;  \left\{q(x) := \sum_{i=1}^{n} 2\alpha_i x_i + \sum_{i=1}^{n} \beta_i x_i^2 \;\middle  \vert\; x \in \Delta^n \right\},
\end{equation}
for scalars $\alpha_i, \beta_i \in \bbR$, and has been studied in~\cite{bomze_separable_2012, dey_convexification_2025-1}. Specifically,~\cite[Conjecture 1]{dey_convexification_2025-1} asserts that the DNN relaxation of~\eqref{eq:StandardSeparableProgram} is always tight. We confirm this below.

\begin{proposition}\label{prop:SeparableDNNTight}
Let $q(x) = \sum_{1\leq i \leq n} 2\alpha_i x_i+ \sum_{i = 1}^{n}\beta_i x_i^2$ for some scalars $\alpha_{i},\beta_i \in \bbR$. Then, the doubly nonnegative relaxation of the standard quadratic programming problem \eqref{eq:StandardSeparableProgram} is tight. 
\end{proposition}

\begin{proof}
Considering Remark~\ref{rem:St-QP}, we obtain the matrix representative of the quadratic form $q$ as 
\[Q = \begin{bmatrix} \beta_1 + 2\alpha_1 & \alpha_1 + \alpha_2 & \dots & \alpha_1 + \alpha_n\\ \alpha_1 + \alpha_2 & \beta_2 + 2\alpha_2 & \dots & \alpha_2 + \alpha_n\\ \vdots & & \ddots & \vdots\\ \alpha_n + \alpha_1 & \alpha_n + \alpha_2 & \dots & \beta_n + 2\alpha_n\end{bmatrix}.\]
If $\pi$ is a permutation such that $\alpha_1\leq \alpha_2\leq \cdots \leq \alpha_n$, then permuting the rows and columns of~$Q$ according to $\pi$ gives a matrix in $\cM_n$. The result therefore follows from Corollary \ref{cor:Mn_DNN_tight}.
\end{proof}

Observation \ref{obs:mainObservation}  is also helpful in identifying tight DNN relaxations for subsets of matrices $\cQ \subseteq \cS^n$ for which $\cQ + \spann E \neq \cQ$. This is demonstrated below for matrices with the block structure described by Remark \ref{ex:BlockSign_SPN}. 

\begin{proposition}
Let $\cB$ be the class of matrices with block structure given by Remark \ref{ex:BlockSign_SPN}. If $Q\in \cS^n$ and $Q -z_{\textup{SPN}}E \in \cB$, then $z_{\textup{SPN}} = z_{\textup{COP}} = z^*$.  
\end{proposition}

\begin{proof}
If $z_{\textup{COP}}$ is the optimal value of the copositive program, then $z_{\textup{COP}} \geq z_{\textup{SPN}}$. In particular, we have that 
$Q-z_{\textup{COP}}E \leq Q-z_{\textup{SPN}}E.$
So, blocks of nonpositive entries in $Q-z_{\textup{SPN}}E$ are also nonpositive in $Q-z_{\textup{COP}}E$. Since subtracting multiples of $E$ does not change the ordering of entries, it follows that $Q-z_{\textup{COP}}E \in \cB$. By Remark \ref{ex:BlockSign_SPN}, this implies that $Q-z_{\textup{COP}}E \in \SPN_n$ and therefore $z_{\textup{COP}} = z_{\textup{SPN}} = z^*$. 
\end{proof}

Finally, we consider the case of optimization over the unit sphere. 

\begin{proposition}
Consider the problem 

\[z^* = \min \left\{ x^\top Q x \;\middle\vert\; 
\|x\| =1, x\geq 0 \right\}.\]

If $\cA$ is a class of matrices satisfying Conditions \ref{itm:GenTheoremUniformRow},  \ref{itm:GenTheoremDeleteNonneg}, and \ref{itm:GenTheoremSchur} of Theorem~\ref{thm:MainTheoremGeneral} and $Q \in \cA$, then $z^*$ can be computed by the following relaxation:

\[z^* = \max \left\{ \lambda \;\vert\; Q-\lambda I \in \SPN_n \right\}.\]
\end{proposition}
\begin{proof}
Conditions \ref{itm:GenTheoremUniformRow},  \ref{itm:GenTheoremDeleteNonneg}, and \ref{itm:GenTheoremSchur} of Theorem \ref{thm:MainTheoremGeneral} do not impose restrictions on the diagonals of matrices. So, if $\cA$ satisfies all three conditions, then $Q \in \cA$ if and only if $Q+\lambda I \in \cA$ for all $\lambda \in \bbR$.
\end{proof}

\section{Sign Pattern Graphs}\label{sec:SignGraphs}

Rescaling by a positive diagonal matrix cannot change the sign of entries. So, it is natural to ask if the existence of a diagonal rescaling and permutation taking a copositive matrix $A$ to a matrix in $\mathcal{M}_n$ depends only on the pattern of signs in the entries of $A$. Proposition \ref{prop:SignGraphThreshold} below provides a necessary condition on the sign patterns of entries of matrices which are elements of $G_n\cM_n$.

To state our result, we need the following notation. Given a matrix $A \in \cS^n$, the graph of $A$ is the graph $\cG(A) = ([n],\{(i,j)\; \vert \; i\neq j, a_{i,j} \neq 0\})$. There are two subgraphs corresponding to the signs of nonzero entries in $A$:
\[\cG_{-}(A) = ([n],\{(i,j) \; \vert \; i\neq j, a_{i,j} < 0\}), \quad \cG_{+}(A) = ([n],\{(i,j) \; \vert \; i\neq j, a_{i,j} > 0\}).\]
That is, $\cG_{-}(A)$ and $\cG_{+}(A)$ both have the vertex set $\{1,2,\ldots, n\}$ and $(i,j)$ is an edge in $\cG_{-}(A)$ if $a_{i,j}<0$ and an edge in $\cG_{+}(A)$ if $a_{i,j} > 0$. 

A simple graph $G = (V,E)$ is a threshold graph \cite{chvatal_aggregation_1977} if it can be constructed by iteratively adding vertices which are either isolated or connected to every other vertex. That is, if $G = (V,E)$ is threshold, then one constructs 

\[\tilde{G} = \left(V\cup \{u\}, E \cup \{(v,u)\; \vert \; v \in V\}\right) \quad \text{ or } \quad \widehat{G} = (V\cup \{u\}, E).\]
The forbidden subgraphs of a threshold graph are the path on 4 vertices, the 4 cycle, and two paths on two vertices \cite{chvatal_aggregation_1977}. A graph is threshold if and only if its complement is threshold. Note that this implies that threshold graphs are perfect, as they cannot contain an odd cycle of length at least 5 as an induced subgraph. 

\begin{proposition}\label{prop:SignGraphThreshold}
Let $A \in G_n\cM_n$. Then, $\cG_{-}(A)$ and $\cG_{+}(A)$ are both threshold graphs. 
\end{proposition}

\begin{proof}
 The base case $n = 1$ is trivial. 
 
Suppose that $n > 1$ and that for all $k \leq n-1$, we have that matrices $B \in G_k\cM_k$ are such that $\cG_{-}(B)$ and $\cG_{+}(B)$ are both threshold.
Let $(PD) \in G_n$ be the product of a permutation matrix and a diagonal matrix with positive entries such that $\tilde{A}:= (PD)^\top A (PD)$ is an element of $\cM_n$. We first show that $\cG_{-}(\tilde{A})$ is threshold.  By Lemma \ref{lem:Uniform_Row}, either the first row of $\tilde{A}$ has all negative off-diagonal entries or the last row of $\tilde{A}$ has all nonnegative off-diagonal entries. In the first case, we have that $\tilde{A}(1) \in \cM_{n-1}$ and $\cG_{-}(\tilde{A})$ is obtained from $\cG_{-}(\tilde{A}(1))$ by adding a vertex connected to all vertices of $\cG_{-}(\tilde{A}(1))$. In the second case, $\tilde{A}(n) \in \cM_{n-1}$ and $\cG_{-}(\tilde{A})$ is obtained from $\cG_{-}(\tilde{A}(n))$ by adding an isolated vertex. By the inductive hypothesis, this implies that $\cG_{-}(\tilde{A})$ is threshold in both cases. 

We now show that $\cG_{+}(\tilde{A})$ is threshold. This is immediate if $A$ has all nonzero entries, since in this case, $\cG_{+}(\tilde{A})$ is the complement of $\cG_{-}(\tilde{A})$ and the complement of a threshold graph is threshold. Otherwise, by Remark \ref{rem:Uniform_Row_Alternate}, either the first row of $\tilde{A}$ has all nonpositive off-diagonal entries or the last row of $\tilde{A}$ has all positive entries. Arguing as above, this implies that $\cG_{+}(\tilde{A})$ is threshold as well. 

Since $\cG_{\pm}(\tilde{A})$ and $\cG_{\pm}(A)$ differ only in labeling of the vertices, the claim follows. 
\end{proof}

The condition on sign graphs in Proposition \ref{prop:SignGraphThreshold} is not sufficient, as demonstrated below in Example \ref{ex:SignPatternsSame}.

\begin{example}\label{ex:SignPatternsSame}
 Consider the following matrices $A$ and $B$:

\[A = \begin{bmatrix} 1 & -\sqrt{2} & -\sqrt{2} & -\sqrt{2} & -\sqrt{2} & -\sqrt{2}\\ -\sqrt{2} & 3 & 1 & 3 & 3 & 1\\
-\sqrt{2} & 1 & 3 & 1 & 3 & 3\\
-\sqrt{2} & 3 & 1 & 3 & 1 & 3\\
-\sqrt{2} & 3 & 3 & 1 & 3 & 1\\
-\sqrt{2} & 1 & 3 & 3 & 1& 3\end{bmatrix} \quad \text{ and } B = \begin{bmatrix} 1 & -\sqrt{2} & -\sqrt{2} & -\sqrt{2}& -\sqrt{2} & -\sqrt{2}\\ -\sqrt{2} & 3 & 3 & 3 & 3 & 3\\
-\sqrt{2} & 3 & 3 & 3 & 3 & 3\\
-\sqrt{2} & 3 & 3 & 3 & 3 & 3\\
-\sqrt{2} & 3 & 3 & 3 & 3 & 3\\
-\sqrt{2} & 3 & 3 & 3 & 3 & 3\end{bmatrix}.\]

The matrices $A$ and $B$ both have the same sign pattern in their entries and $B \in \mathcal{M}_6$. The sign pattern graphs of $A$ and $B$ are displayed in Figure \ref{fig:SignPatternEx} and are both threshold.

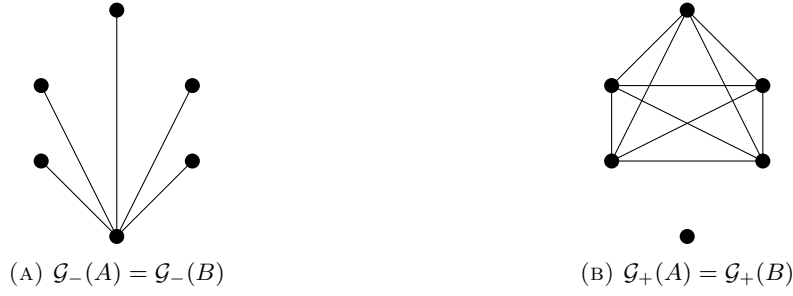
\begin{figure}
\begin{subfigure}[b]{0.45\textwidth}
\centering
\begin{tikzpicture}
    \node[circle, fill, inner sep= 2pt] at (1,0) {};
    \node[circle, fill, inner sep = 2pt] at (2,1) {};
    \node[circle, fill, inner sep = 2pt] at (2,2) {};
    \node[circle, fill, inner sep = 2pt] at (1,3) {};
    \node[circle, fill, inner sep = 2pt] at (0,2) {};
    \node[circle, fill, inner sep = 2pt] at (0,1) {};

    \draw (1,0)--(2,1);
    \draw (1,0)--(2,2);
    \draw (1,0)--(1,3);
    \draw (1,0)--(0,2);
    \draw (1,0)--(0,1);
\end{tikzpicture}
\caption{$\cG_{-}(A) = \cG_{-}(B)$}
\end{subfigure}
\begin{subfigure}[b]{0.45\textwidth}
\centering
    \begin{tikzpicture}
    \node[circle, fill, inner sep= 2pt] at (1,0) {};
    \node[circle, fill, inner sep = 2pt] at (2,1) {};
    \node[circle, fill, inner sep = 2pt] at (2,2) {};
    \node[circle, fill, inner sep = 2pt] at (1,3) {};
    \node[circle, fill, inner sep = 2pt] at (0,2) {};
    \node[circle, fill, inner sep = 2pt] at (0,1) {};

    \draw (2,1)--(2,2);
    \draw (2,1)--(1,3);
    \draw (2,1)--(0,2);
    \draw (2,1)--(0,1);

    \draw (2,2)--(1,3);
    \draw (2,2)--(0,2);
    \draw (2,2)--(0,1);

    \draw (1,3)--(0,2);
    \draw (1,3)--(0,1);

    \draw (0,2)--(0,1);
\end{tikzpicture}
\caption{$\cG_+(A) = \cG_+(B)$}
\end{subfigure}
\caption{Sign pattern graphs for the matrices $A$ and $B$ in Example \ref{ex:SignPatternsSame}. Both graphs are threshold, but $A \not \in \cM_6$.}\label{fig:SignPatternEx}
\end{figure}

However, the Schur complement of $A$ with respect to the first row is 

\[A/a_{1,1} = \begin{bmatrix} 3 & 1 & 3 & 3 & 1\\
 1 & 3 & 1 & 3 & 3\\
  3 & 1 & 3 & 1 & 3\\
  3 & 3 & 1 & 3 & 1\\
  1 & 3 & 3 & 1& 3\end{bmatrix} - 2E_5 = \begin{bmatrix}1 & -1 & 1 & 1 & -1 \\ -1 & 1 & -1 & 1 & 1\\ 1 & -1 & 1 & -1 & 1\\ 1 & 1 & -1 & 1 & -1\\ -1 & 1 & 1 & -1 & 1 \end{bmatrix},\]
 where $E_5$ is the $5\times 5$ all ones matrix. Since this is the Horn matrix, which is in $\Cop_5 \setminus \SPN_5$, this implies that $A \in \Cop_6 \setminus \SPN_6$ by Lemma \ref{lem:Uniform_Row}. So, no diagonal scaling of $A$ can be an element of $\mathcal{M}_6$. 
\end{example}

\subsection{Comparison with Known Results for Perfect Graphs}

In \cite{gokmen2022standard}, the authors prove that the DNN relaxation of a standard quadratic programming problem is tight for a class of problems related to perfect graphs. 

\begin{theorem}[{\cite[Proposition 6]{gokmen2022standard}}]\label{thm:PerfectGraph}
For a perfect graph $G = ([n],E)$ on $n$ vertices and a strictly positive weight vector $w \in \bbR^n$, let 

\[\begin{split}
\cA(G,w) = \{B = (b_{ij})\in \cS^n \; \vert \;& b_{i,i} = w_k \text{ for all } k \in [n], \ b_{i,j} = 0 \text{ for all } (i,j) \not \in E, \text{ and } \\ & 2b_{i,j} \geq w_i + w_j \text{ for all } (i,j) \in E\}.
\end{split}\]
If $Q \in \bigcup_{G \text{ perfect}}\bigcup_{w \in \bbR^n_{++}} \cA(G,w) + \spann E$, then $z^* = z_{SPN}$.
\end{theorem}

The set of matrices $\cA(G,w)$ has relevance in graph theory. If $\omega(G,\frac{1}{w})$ is the maximum weighted clique in the graph $G$ with vertex weights $\frac{1}{w_i}$ and $Q \in \cA(G,w)$, then the optimal value of the standard quadratic program with objective function $Q$ satisfies $z^* = z_{\textup{COP}} = \frac{1}{\omega(G,\frac{1}{w})}$ \cite{gibbons1997continuous}. It therefore follows that $z_{\textup{COP}} >0$. On the other hand, since copositive matrices have nonnegative diagonals, it must be the case that $z_{\textup{COP}} \leq \min_{i}w_i$. In particular, $Q-z_{\textup{COP}}E$ has negative entries when $(i,j) \not \in E$ and positive entries when $(i,j) \in E$.

From this perspective, Theorem \ref{thm:PerfectGraph} gives a condition under which copositivity implies the existence of an SPN decomposition when the sign graphs are perfect. This condition relies on the inequality $2Q_{i,j} \geq Q_{i,i} + Q_{j,j}$, which relates the diagonal and off-diagonal entries of the objective function matrix. 

Our results, on the other hand, provide conditions under which copositivity implies the existence of an SPN decomposition which apply when the sign graphs are threshold. This is a strict subset of perfect graphs. However, the conditions on matrix entries are less restrictive, as they do not involve the diagonal of the matrix.

\printbibliography

\end{document}